\documentclass{amsart}
\usepackage{graphicx}

\usepackage{latexsym,amscd,amssymb,epsfig}

\usepackage{epstopdf}

\theoremstyle{plain}
   \newtheorem{theorem}{Theorem}[section]
   \newtheorem{proposition}[theorem]{Proposition}
   \newtheorem{lemma}[theorem]{Lemma}
   \newtheorem{corollary}[theorem]{Corollary}
   \newtheorem{conjecture}[theorem]{Conjecture}
\theoremstyle{definition}
   \newtheorem{definition}[theorem]{Definition}
   \newtheorem{example}[theorem]{Example}

   \newtheorem{remark}[theorem]{Remark}

\numberwithin{equation}{section}

\newcommand\sym{{\mathfrak{S}}}
\newcommand\sh{{\mathrm{shape}}}
\newcommand\len{{\mathrm{length}}}
\newcommand\shf{{\mathrm{shf}}}
\newcommand\st{{\mathrm{std}}}

\newcommand\kc{{\mathcal{Y}}}
\newcommand\Knuth{{\underset{K}{\sim}}}
\newcommand\nuth{{\underset{K^*}{\sim}}}
\newcommand\Inv{{\mathrm{Inv}}}
\newcommand\Des{{\mathrm{Des}}}
\newcommand\Par{{\mathrm{Par}}}

\newcommand{\cellsize}{16}
\newlength{\cellsz} 
\newsavebox{\cell}
\sbox{\cell}{\begin{picture}(\cellsize,\cellsize)
\put(0,0){\line(1,0){\cellsize}} \put(0,0){\line(0,1){\cellsize}}
\put(\cellsize,0){\line(0,1){\cellsize}}
\put(0,\cellsize){\line(1,0){\cellsize}}
\end{picture}}
\newcommand\cellify[1]{\def\thearg{#1}\def\nothing{}%
\ifx\thearg\nothing \vrule width0pt height\cellsz depth0pt\else
\hbox to 0pt{\usebox{\cell} \hss}\fi%
\vbox to \cellsz{ \vss \hbox to \cellsz{\hss$#1$\hss} \vss}}
\newcommand\tableau[1]{\vtop{\let\\\cr
\baselineskip -16000pt \lineskiplimit 16000pt \lineskip 0pt
\ialign{&\cellify{##}\cr#1\crcr}}}

\begin{document}
\title[The Weak Order] {Inner tableau translation property of the weak order and related results}

\author{ M\"{U}GE  TA\c{S}K{I}N }
\address{Bogazici Universitesi, Istanbul, Turkey}

\thanks{This research forms a part of the author's doctoral thesis at the Univ. of
Minnesota, under the supervision of Victor Reiner, and partially
supported by NSF grant DMS-9877047.}

\begin{abstract}
Let $SYT_{n}$ be the set of all standard Young tableaux with $n$
cells and $\leq_{weak}$  be Melnikov's  the weak order on $SYT_n$.  The aim of this paper is to introduce  a conjecture
 on the weak order, named the {\it property of inner tableau translation}, and  discuss its  significance. We will  also prove the
 conjecture for some  special cases.\end{abstract}

\maketitle

\section{Introduction}

The weak  order first introduced  by Melnikov  and well studied in
\cite{Melnikov1, Melnikov2,Melnikov3} due its strong connections to
Kazhdan-Lusztig  and geometric order on standard Young tableaux,
where the latter are induced from the representation theory of special linear
algebra and symmetric group. We have  the following inclusion among
all of these  three orders:
$$
\text{ weak order } \subsetneq \text{ Kazhdan-Lusztig (KL) order }
\subseteq \text{ geometric order } .
$$
The fact that its definition  uses just  the  combinatorics of tableaux such as   Knuth relations and the weak order on symmetric group,
 gives the weak order an important place among these orders. On the other hand the only justification for  its well-definedness
is induced from above inclusion, in other words there is no self
contained  proof of this basic fact.

Our aim here is to  bring the attention to the following  conjecture
on  the weak order (which is first asked  in \cite{Taskin}), called  the property of {\it inner
tableau translation}.  This property  is known to be satisfied by Kazhdan-Lusztig  and geometric orders  and its  importance on the weak order  relies on
the fact that it provides  a self contained proof for the
well-definedness of this order.

\begin{conjecture} Given two tableau $S<_{weak}T$  having  the same inner tableau $R$,   replacing $R$ with
 another same shape tableau $R'$   in $S$ and $T$   still preserves the weak order.
 \end{conjecture}

In the following  we  first provide the  definitions and related
background for the weak order. In the third section, by assuming the
conjecture  we will provide a self contained the proof for  the well
definedness  the weak order and we close this section with the
discussion on how this conjecture plays specific role in studies of
Poirier and Reutenauer Hopf algebra on standard Young tableaux. In
the last section, we prove the conjecture for the case when the
inner tableau $R$ has hook shape or a shape which  consists of two
rows or two columns.

\section{Related background}

\subsection{Definition of the weak order}
The definition of the weak order uses well known  Robinson-Schensted
$(RSK)$ correspondence which bijectively assigns to every
permutation $w\in S_n$ a pair of same  shape tableaux $(I(v), R(w))
\in SYT_{n}\times SYT_{n}$, where  $I(w)$ and $R(w)$ are called the
{\it insertion} and {\it recording tableau} of $w$ respectively. On
the other hand  an   equivalence relation $\Knuth$ on $S_n$ due to
Knuth \cite{Knuth1}  plays a crucial role in this correspondence.
Namely:
$$ u~\Knuth~ w \iff I(u)=I(w).$$ We will denote the corresponding equivalence classes in
$S_n$ by $\{\mathcal{Y}_T\}_{T\in SYT_n}$.

Let us explain these algorithms briefly. Denote by
$(I_{i-1},R_{i-1})$ the same shape tableaux obtained by insertion
and recording algorithms on the first $i-1$ indices of $w=w_1\ldots
w_n$. In order to get $I_i$, if $w_{i}$ is greater then the last
number on the first row of $I_{i-1}$, it is concatenated to the
right side of the first row of $I_{i-1}$, otherwise, $w_i$ replaces
the smallest number, say $a$ among all numbers in the first row
greater then $w_i$ and this time insertion algorithm is applied to
$a$ on the next row. Observe that after finitely many  steps the
insertion algorithm terminates with a new added cell. The resulting
tableau is then $I_i$ and recording tableau $R_i$ is found by
filling this new cell in $R_{i-1}$ with the number $i$. We
illustrate these algorithms with the following example

\begin{example}
 Let $w=52413$. Then,
$$\begin{aligned}
&I_1=5 &\Rightarrow
I_2=\begin{array}{c}2\\5\end{array}&\Rightarrow
I_3=\begin{array}{cc}2&4\\5\end{array}&\Rightarrow
I_4=\begin{array}{cc}1&4\\2\\5\end{array}&\Rightarrow
I_5=\begin{array}{cc}1&3\\2&4\\5\end{array}=I(w)\\
&R_1=1&\Rightarrow
R_2=\begin{array}{c}1\\2\end{array}&\Rightarrow
R_3=\begin{array}{cc}1&3\\2\end{array}&\Rightarrow
R_4=\begin{array}{cc}1&3\\2\\4\end{array}&\Rightarrow
R_5=\begin{array}{cc}1&3\\2&5\\4\end{array}=R(w)
\end{aligned}
$$
\end{example}

\begin{definition}
We say  $u, w \in \sym_{n}$ differ by one {\it Knuth relation},
written $u \stackrel{K}{\cong} w$, if $$ \begin{array}{cl}
\mbox{either }  & w=x_{1} \ldots yxz \ldots x_{n} \mbox{ and }
          u=x_{1} \ldots yzx \ldots x_{n} \\
 \mbox{or }      & w=x_{1} \ldots xzy \ldots x_{n} \mbox{ and }
          u=x_{1} \ldots zxy \ldots x_{n}
\end{array} $$
for some  $ x<y<z$. Two permutations are called  {\it Knuth
equivalent}, written $u  \stackrel{K}{\cong}w$, if there is a
sequence of permutations such that
\[ u=u_{1}  \stackrel{ K}{\cong} u_{2} \ldots
\stackrel{K}{\cong} u_{k}= w.\]
\end{definition}

Sch\"utzenberger's {\it jeu de taquin}  slides
\cite{Schutzenberger2} are one of the combinatorial operations on
tableaux that we apply often in the following sextions.

 \begin{definition}
Let $\lambda=(\lambda_{1},\lambda_{2}, \ldots,\lambda_{k}) $ and
$\mu=(\mu_{1},\mu_{2}, \ldots,\mu_{l}) $ be two Ferrers diagrams
such that $\mu \subset \lambda$. Then the corresponding {\it  skew
diagram} is defined to be the set of cells
$$ \lambda/\mu=\{ c : c \in \lambda, c \notin \mu\}. $$
A skew diagram is called {\it normal} if $\mu=\varnothing$. A {\it
partial skew tableau} of shape $\lambda/\mu$ is an array of distinct
integers elements whose rows and columns increase. A {\it standard
skew tableau} of shape $\lambda/\mu$ is partial skew tableau whose
elements are $\{1,2,\ldots,n\}$.
\end{definition}

 We next  illustrate  the  forward and backward slides of
 Sch\"utzenberger's {\it jeu de taquin} \cite{Schutzenberger2} without the definition.

\begin{example}
Let  $P=\begin{array}{ccc}&&4 \\&2&5  \\1& 3 &
\\\end{array}$ and $ Q=\begin{array}{ccc} &  2 & 4 \\& 3&  5
\\1&&\\ \end{array}$. Below we illustrate a forward and backward slide on $P$ and $Q$ through the cells indicated by dots.

$$ \begin{array}{ccc}& \bullet  & 4 \\& 2 & 5  \\1& 3 &
\\\end{array}\rightarrow
\begin{array}{ccc}& 2  & 4 \\& \bullet & 5  \\1& 3 &  \\\end{array}\rightarrow
\begin{array}{ccc}&  2 & 4 \\& 3 &  5\\1& \bullet  &  \\\end{array}=
\begin{array}{ccc}&  2 & 4 \\& 3&  5 \\1&  &
\\\end{array}$$

$$ \begin{array}{ccc} &
2 & 4 \\& 3 &  5\\1& \bullet  &
\\\end{array} \rightarrow
\begin{array}{ccc} & 2  & 4 \\& \bullet & 5  \\1& 3 &   \end{array} \rightarrow
\begin{array}{ccc} & \bullet  & 4 \\& 2 & 5  \\1& 3 &  \end{array} =
\begin{array}{ccc} &   & 4 \\ & 2 & 5  \\ 1& 3 &  \\
\\ \end{array}$$
\end{example}

The other main ingredient of the weak order is  the ({\it right})
{\it weak} {\it Bruhat order}, $\leq_{weak}$, on  $S_n$ which
obtained by taking the  transitive closure of the following
relation:
 $$u \leq_{weak} w  ~\text{if}~ w=u \cdot s_i  \text{ and } \len(w)=\len(u)+1$$
 where $s_i$ denotes  the adjacent transposition $(i, i+1)$ and $\len(w)$ measures the size of a reduced word of $w$.  The weak order has an alternative characterization
\cite[Prop. 3.1]{Bjorner2}  in terms of {\it (left) inversion sets}
namely
$$u \leq_{weak} w ~\text{if and only if} ~~ \Inv_L(u)\subset
\Inv_L(w)$$ where $\Inv_L(u):=\{(i,j): 1 \leq i < j \leq n ~\text{
and }~
         u^{-1}(i) > u^{-1}(j)\}$.

\begin{definition}
\label{weak-order-def}
 The {\it weak order} $(SYT_n,\leq_{weak})$,
first introduced by Melnikov \cite{Melnikov1} under the name {\it
induced Duflo order}, is the partial order induced by taking
transitive closure of  the following relation:
$$
\begin{aligned}
S\leq_{weak} T & \mbox{ if there exist } ~\sigma\in \kc_S, ~\tau\in
\kc_T
                \mbox{ such that } ~\sigma\leq_{weak} \tau.
\end{aligned}
$$
\end{definition}

 The necessity of taking the transitive closure in the definition of the
 weak order is illustrated by the following example (cf.
Melnikov \cite[Example 4.3.1]{Melnikov1}).

\begin{example}
Let $R=\scriptstyle{\begin{array}{ccc} 1& 2 & 5 \\  3& 4 &
\end{array}}$,
\hskip .05in $S=\scriptstyle{\begin{array}{ccc} 1& 4 & 5 \\  2&  &
\\ 3 & & \end{array}}$,
\hskip .05in $T=\scriptstyle{
\begin{array}{ccc} 1 & 4 &  \\  2 & 5 &  \\ 3 & &\end{array}}$
with
$$
 \begin{array}{ll}
       &  \mathcal{Y}_{R}=\{ 31425, 34125, 31452, 34152, 34512 \}, \\
       & \mathcal{Y}_{S}=\{ 32145, 32415, 32451, 34215, 34251 ,34521 \}, \\
       & \mathcal{Y}_{T}= \{ 32154, 32514, 35214, 32541, 35241\}. \\
 \end{array}
 $$
Here $R <_{weak} S $ since $34125 <_{weak} 34215$, and $S <_{weak} T
$ since $32145 <_{weak} 32154$. Therefore $ R <_{weak} T $. On the
other hand, for every $ \rho \in \mathcal{Y}_{R}$ and for every
$\tau \in \mathcal{Y}_{T}$ we have $(2,4) \in \Inv_L(\rho)$ but
$(2,4) \notin \Inv_{L}(\tau)$.
\end{example}

\begin{figure}
\begin{center}
 \includegraphics[scale=0.45]{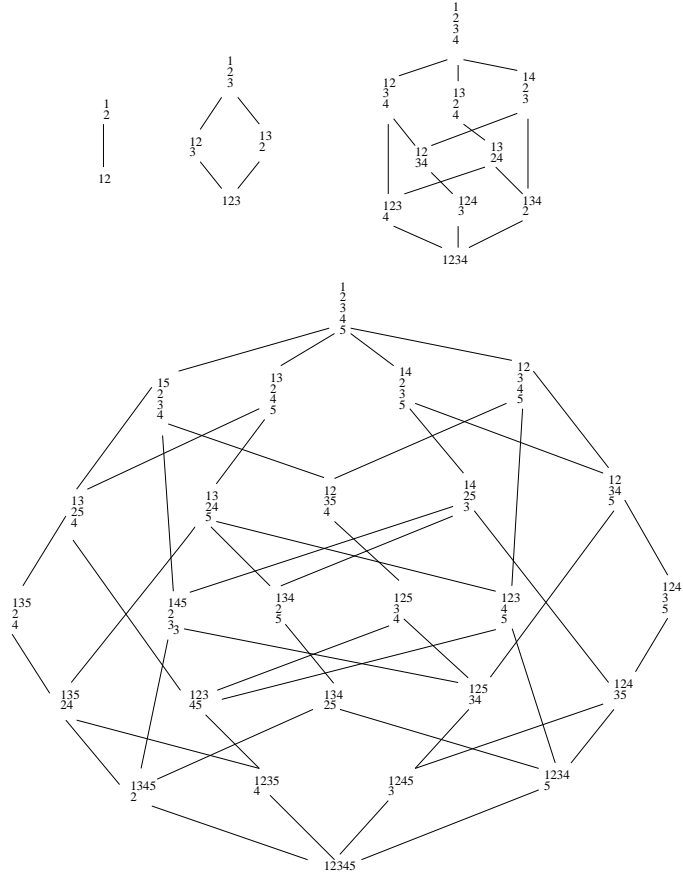} \caption{ \label{figure1}
 The weak  on $SYT_{n}$ for $n = 2,3,4,5$.}\end{center}
\end{figure}

\subsection{Some basic properties of the weak order} For $u\in S_n$  and  $1\leq i< j \leq n$,  let  $u_{[i,j]}$ be the
word obtained by  restricting $u$ to the segments $[i,j]$  and
$\st(u_{[i,j]}) \in S_{j-i+1}$ be the permutation  obtained from
$u_{[i,j]}$ by subtracting $i-1$ from each letter.

Similarly for $S\in SYT_n$ and  $1\leq i< j \leq n$,  let
$S_{[i,j]}$ be the normal shape  tableau obtained by  restricting
$S$ to the segments $[i,j]$ and  by applying  Sch\"utzenberger's
back word jeu-de-taquin slides. Then $\st(S_{[i,j]}) \in
SYT_{j-i+1}$ be the tableau obtained from $S_{[i,j]}$ by subtracting
$i-1$ from each letter.

 In fact $\Inv_L(u)\subset \Inv_L(w)$ gives
$\Inv_L(u_{[i,j]})\subset \Inv_L(w_{[i,j]})$ for all $1\leq i < j
\leq n$ and  hence
\begin{equation}\label{welldefined-weak-order1}
u \leq_{weak} w ~\mbox{ implies } ~\st(u_{[i,j]}) \leq_{weak} \st(w_{[i,j]})
~\mbox{ for all } ~1\leq i<j\leq n.
\end{equation}

The following   basic fact about $RSK$, Knuth equivalence, and
jeu-de-taquin  are essentially due to Knuth and Sch\"utzenberger;
see Knuth \cite[Section 5.1.4]{Knuth} for detailed explanations.

\begin{lemma}
\label{j-d-t-initial-final} Given $u \in S_n$,
 let $I(u)$ be the insertion tableau of $u$.  Then for $1\leq i<j \leq
 n$,
$$
\st(I(u)_{[i,j]}) = I(\st(u_{[i,j]})).$$
\end{lemma}

Therefore we have following:

\begin{lemma}
\label{inner.restriction.lemma}The weak order {\it restricts to
 segments}, i.e.,
 $$S \leq T~~\text{ implies }~~\st(S_{[i,j]}) \leq
 \st(T_{[i,j]})~~ \text{ for all } ~~1\leq i<j \leq n.
 $$
\end{lemma}

\begin{remark}
Melnikov shows in  \cite[Page 45]{Melnikov1} that the geometric
order also restricts to segments. On the other hand the   same fact
about Kazhdan-Lusztig order  was first shown by Barbash and Vogan
\cite{Barbash-Vogan} for arbitrary finite Weyl groups (see also work
by Lusztig \cite{Lusztig2}) whereas the generalization to Coxeter
groups  is due to Geck \cite[Corollary 3.4]{Geck}.
\end{remark}

Now recall that {\it (left) descent set} of a permutation $\tau$ is
defined by
$$
\Des_{L}(\tau) :=\{i: 1 \leq i \leq n-1 ~\text{ and }~
       \tau^{-1}(i) > \tau^{-1}(i+1)  \}. \\
$$
On the other hand the {\it descent set} of the standard Young
tableau $T$ is described intrinsically by
$$
\begin{aligned}
\Des(T) &:=
   \{(i,i+1): 1 \leq i \leq n-1 \text{ and }  \\
   & \qquad \qquad  i+1 \text{ appears in a row below } i \text{ in }T\}.
\end{aligned}
$$

As a consequence of a well-known properties of $RSK$ we have the
following basic fact:

\begin{lemma}
\label{descent.lemma0}
 For any $\tau \in \mathcal{Y}_{T}$ we have
$$ \Des_{L}(\tau)=\Des(T)
$$
i.e., the left descent set  is constant on Knuth classes.
\end{lemma}

We let $(2^{[n-1]}, \subseteq)$ be the Boolean algebra of all
subsets of $[n-1]$ ordered by inclusion. \vskip .2in

\begin{lemma}
\label{descent.lemma}
Let $\leq$ be any order on $SYT_n$ which is stronger than the weak
order and restricts to segments. Then the map
$$(SYT_n, \leq) \mapsto (2^{[n-1]}, \subseteq)$$
sending any tableau $T$ to its descent set $\Des(T)$ is order
preserving.
\end{lemma}

We denote by $(\Par_n, \leq^{op}_{dom})$ the set of all partitions
of the number $n$ ordered by the {\it opposite (or dual) dominance
order}, that is, $\lambda \leq^{op}_{dom} \mu$ if
$$
\lambda_1 + \cdots + \lambda_k \geq \mu_1 + \cdots + \mu_k \text{
for all } k.
$$

The following can be easily deduced from  Greene's theorem \cite{Greene}.

\begin{lemma}
\label{shape.lemma} $S\leq_{weak} T$ implies $\sh(T) \leq_{dom} \sh(S) $
\end{lemma}

Recall that for a standard young tableau $T$, $T^t$ denotes the
transpose of $T$ whereas  $T^{evac}$  denotes the tableau found by
applying
 the Sch\"utzenberger's \cite{Schutzenberger1} evacuation map on $T$.  For any $\tau=\tau_1\tau_2\ldots\tau_n \in \mathcal{Y}_{T}$  we have
$$\begin{aligned}
&\tau^t=\tau_{n}\tau_{n-1}\ldots\tau_1 \in \mathcal{Y}_{T^t}\\
&\tau^{evac}=(n+1-\tau_{n})(n+1-\tau_{n-1})\ldots(n+1-\tau_1) \in \mathcal{Y}_{T^{evac}}
\end{aligned}$$

\begin{proposition}
\label{order.preserving.maps} Suppose $S\leq_{weak}T$ in $SYT_n$.
Then
\begin{enumerate}
\item  $S^{evac} \leq_{weak} T^{evac}$.
\item  $T^t \leq_{weak}S^t$.
\end{enumerate}
\end{proposition}

\begin{proof} Let $w_0$ be the longest element in $S_n$.  Then  the maps
$$
 w \mapsto w_0w ~\text{ and }~w \mapsto ww_0
 $$
are clearly  anti-automorphisms and hence   $w \mapsto w_0ww_0$  is
a automorphism of   $(S_n,\leq_{weak})$. On the other hand $I(ww_0)$
is just the  transpose tableau of $I(w)$ \cite{Schensted} whereas
$I(w_0ww_0)$ is nothing but the evacuation of $I(w)$
\cite{Schutzenberger1}.
\end{proof}

%%%%%%%%%%%%%%%%%

\subsection{Inner tableau translation property}
The   {\it dual Knuth} relations $\nuth$ on $S_n$ plays the main
role in the definition of inner tableau translation property. In its most
basic form  this relation is defined  through the Knuth relations
applied on the inverse of permutations. Namely,
$$\sigma ~\nuth~\tau  ~ \text{in}~ S_n ~ \text{if and olny if}~ \sigma^{-1} ~\Knuth~
\tau^{-1}.$$ An equivalent definition can be given by taking the
transitive closure of  the following:  We say $ \sigma$ and $\tau$
differs by  a single dual Knuth relation determined by the triple
$\{i, i+1, i+2\}$  if
$$\begin{aligned}
\mbox{ either } & \sigma=\ldots i+1 \ldots i \dots i+2 \ldots \mbox{
and }
\tau=\ldots i+2 \ldots i \dots i+1 \ldots \\
\mbox{ or } & \sigma=\ldots i+1 \ldots i+2 \dots i \ldots \mbox{ and
} \tau=\ldots i \ldots i+2 \dots i+1 \ldots
\end{aligned}$$

Since left descent sets are all equal for the  permutations lying in
the same Knuth class,  the  dual Knuth relation defines an action  on the standard Young tableaux. In order to present this action let us give the following definition.

\begin{definition} For  $S\in SYT_n$ let  $A=\{(i,j)\} $ be a cell  lying in  $\sh(S)$, where $i$ denotes the row number counted
from the top and $j$ denotes the column number counted   from the
left. Then
$$\begin{aligned}(S,A,\mathrm{ne}):=&\{ (k,l) ~\mid~  k<i ~\text{and}~ l\geq j
\}\\
(S,A,\mathrm{sw}):=&\{ (k,l) ~\mid~  k\geq i ~\text{and}~ l<j \}
\end{aligned}
$$
\end{definition}

Suppose that $\sigma \in \kc_S$ has $i\in\Des(\sigma)$ but
$i\not\in\Des(\sigma)$. Therefore  $\sigma$ has one of the following
form
 $$\sigma=\ldots i+1 \ldots i \dots i+2 \ldots ~~\text{or}~~\sigma=\ldots i+1 \ldots i+2\dots i\ldots.$$
Now denote by $C_i, C_{i+1}$ and $C_{i+2}$ the cells labeled by
$i,i+1$ and $i+2$ in $S$, respectively.  Then
$$ \text{ either }~  C_i \in (S,C_{i+1},\mathrm{ne})\cap (S,C_{i+2},\mathrm{sw}) ~\text{ or }~C_{i+2} \in (S,C_{i+1},\mathrm{ne})\cap
 (S,C_{i},\mathrm{sw})$$
 and the action of a single  dual Knuth relation determined by the triple
$\{i, i+1, i+2\}$ on $S$  interchanges the places of  $i+1$ and
$i+2$ in the fist case and  it interchanges
  the places of $i$ and $i+1$ in the second case.

The following theorem (see \cite[Proposition 3.8.1]{Sagan}) provides
an important  characterization of the dual Knuth relation.

\begin{proposition}
\label{dual.Knuth}
 Let $S, T \in SYT_n$. Then
$ S \nuth T ~~\text{ if and only if }~~ \sh(S)=\sh(T).$
\end{proposition}

\begin{definition}
Let $\{\alpha,\beta\}=\{i,i+1\}$ and
$SYT_n^{[\alpha,\beta]}:=\{ T \in SYT_n \mid
\alpha \in \Des(T), \beta \not \in \Des(T)\}.$
Then we have  {\it inner
translation map}
$$ \mathcal{V}_{[\alpha,\beta]}: SYT_n^{[\alpha,\beta]}\mapsto
SYT_n^{[\beta,\alpha]}$$ which send every tableau $T\in
SYT_n^{[\alpha,\beta]}$ to a tableau obtained as a result of the
action of the  single dual Knuth relation determined by the triple
$\{i, i+1, i+2\}$.
\end{definition}

The inner translation map is first introduced  by Vogan in ~\cite{Vogan} where he also shows that   Kazhdan-Lusztig order is preserved under this map.  For  geometric
order this result is  due to Melnikov ~\cite[Proposition
6.6]{Melnikov4}. On the other hand the example given below shows that the weak order does not
satisfy this property.

\begin{example}
$$\begin{array}{ccc} 1&2&\mathbf{4}\\\mathbf{3}&\mathbf{5}&6 \end{array}
\leq_{weak}
\begin{array}{ccc} 1&2&\mathbf{4}\\\mathbf{3}&6\\\mathbf{5} \end{array}
\hskip .1in \text{ but } \hskip .1in
\begin{array}{ccc} 1&2&\mathbf{3}\\\mathbf{4}&\mathbf{5}&6
\end{array} \not\leq_{weak}
\begin{array}{ccc} 1&2&\mathbf{5}\\\mathbf{3}&6\\\mathbf{4} \end{array}$$
where the latter pair is obtained from the former by applying a
single dual Knuth relation on the triple $\{3,4,5\}$.
\end{example}

A weaker version of the inner translation property can be defined in
the following manner:

\begin{definition}  For  $1 \leq k<n$ and $R\in SYT_k$ let $SYT_n^{R}:=\{ T  \in SYT_n \mid T_{[1,k]}=R\}$.
Then for $R,R'\in SYT_k$, having the same shape, we have    {\it inner tableau translation map}
$$ \mathcal{V}_{[R,R']}: SYT_n^{R}\mapsto
SYT_n^{R'}$$
which send every $T\in SYT_n^R$ to the tableau $T'$ obtained by replacing $R$ with $R'$.
\end{definition}

As a consequence of  Proposition~\ref{dual.Knuth}, one can generate
$T'$ by  a  sequence of dual Knuth relations applied on the
subtableau $R$ of $T$. Therefore   if a partial order is preserved
under inner translation map then it is also preserved under inner
tableau translation map. Hence  Kazhdan-Lusztig and geometric orders
have this property. On the other hand it is still reasonable to ask
whether the weak order is preserved under the inner tableau
translation property.

\begin{conjecture}
\label{main.conjecture}  Let  $S\lessdot_{weak}T$  be a covering
relation in $ SYT_n^{R}$ and $R'$ be a tableau obtained by  applying
to $R$ a single  dual Knuth relation. Then
 $$\mathcal{V}_{[R,R']}(S)\lessdot_{weak}\mathcal{V}_{[R,R']}(T) ~~\text{in}~~ SYT_n^{R'}.$$
  In other words the weak order on standard Young tableau is preserved
under the inner tableau translation map.
\end{conjecture}

\begin{remark} Recall that  any tableau $R' \in SYT_r$  with the same shape as $R$,
 can be obtained by applying to $R$ a sequence of dual Knuth relation by Proposition\ref{dual.Knuth}.
 Therefore one can generalize the conjecture for any tableau $R$ and $R'$  having the same shape.

\end{remark}

As it is stated earlier this conjecture is checked by computer programing up to $n=9$. In the last section we also show that for a specific case the conjecture is true.

\section{Applications of the conjecture}

\subsection{Well-definedness of the weak order}

 By assuming Conjecture~\ref{main.conjecture}, we first prove the following
 result.

 \begin{theorem} The weak order on $SYT_n$ is well defined.
\end{theorem}

\begin{proof} It is enough to show that if $S\leq_{weak} T$ and $S\not = T$ then   $T\not\leq_{weak}S$.
By   Lemma~\ref{descent.lemma}  we know that $S\leq_{weak} T$
implies $\Des(S)\subset\Des(T)$ and  if $\Des(T)-\Des(S)\not =
\emptyset$ then clearly $T\not\leq_{weak}S$.  Now we suppose that
 $\Des(T)=\Des(S)$. Let $k$ be the smallest
integer satisfying
$$S_{[1,k]}=T_{[1,k]}.$$
So $k<n$ and   $S_{[1,k+1]}$ and $T_{[1,k+1]}$ differ
only by the position of the corner  cells labeled by $k+1$.  On the other hand   by
Lemma~\ref{inner.restriction.lemma} have
$$S_{[1,k+1]}\lneq_{weak}T_{[1,k+1]}$$ and Lemma~\ref{shape.lemma}
together with the fact that $\sh(T_{[1,k]}) = \sh(S_{[1,k]})$ gives
\begin{equation}\label{welldefined.eq} \sh(T_{[1,k+1]}) \lneq_{dom} \sh(S_{[1,k+1]}).
\end{equation}

Let $A=\{(i,j)\}$ and $B=\{(i',j')\}$  denote  the cells  labeled by
$k+1$ in $S_{[1,k+1]}$ and $T_{[1,k+1]}$  respectively. Then
\eqref{welldefined.eq} implies that
$$i<i' ~\text{ and } ~j>j'$$  i.e., the corner cell $B$ lies  below  the corner cell $A$ and  therefore there exists a corner cell
  $C=(i'',j'')$ of $S_{[1,k]}=T_{[1,k]}$ which satisfies
$$ i \leq i''< i' ~\text{ and }~ j >j'\geq j''.
$$

Let $R$ denote  $S_{[1,k]}=T_{[1,k]}$ and let $R'$ be another
tableau in $ SYT_k$  such that $\sh(R)=\sh(R')$ and  the corner cell
$C$ of $R'$ is labeled by $k$.    Denote also by $S'$ and $T'$ the
tableaux obtained by replacing $R$ with  $R'$ in $S$ and $T$
respectively. Now by  Conjecture~\ref{main.conjecture} we have
$S'\lessdot_{weak}T'$ and moreover
$$k \in \Des(T')-\Des(S').
 $$
The last argument shows that $T'\not\leq_{weak}S'$ and therefore  by
Conjecture~\ref{main.conjecture} $T\not\leq_{weak}S$.
\end{proof}

\subsection{Poirier-Reutenauer Hopf algebra on
 $\mathbb{Z}SYT= \oplus_{n \geq 0}\mathbb{Z}SYT_{n}$}

Following the work of Malvenuto and Reutenauer on permutations \cite{Malvenuto-Reutenauer}, Poirier-Reutenauer construct two graded Hopf algebra structures on $\mathbb{Z}$ module of
all plactic classes $\{PC_{T}\}_{T \in SYT}$,   where $PC_{T}=\sum_{\footnotesize P(u)=T}u$.  The
product structure of the one that concerns us here is given by
\begin{equation} \label{P-R-multiplication}
PC_{T} \ast PC_{T'} =
        \sum_{\substack{P(u)=T \\  P(w)=T'}}
    \shf(u, \overline{w })
\end{equation}
where $\overline{w}$ is obtained by increasing the indices of $w$ by
the length of $u$ and $\shf$ denotes the shuffle product. Then the bijection sending each plactic class to its defining
tableau gives us a Hopf algebra structure on the $\mathbb{Z}$ module
of all standard Young tableaux, $\mathbb{Z}SYT = \oplus_{n \geq 0}
\mathbb{Z}SYT_{n}$.

In \cite{Poirier-Reutenauer} Poirier and Reutenauer explain this
product using jeu de taquin slides. Following an  analogous result
of Loday and Ronco \cite[Thm. 4.1]{Loday-Ronco} on permutations, the
author  shows the following result in \cite{Taskin}: For  $S\in
SYT_k$, $T\in SYT_l$ where $k+l=n$, let  $\overline{T}$ the tableau
which is obtained by increasing the indices of $T$ by $k$.  Denote
by $S/T$
  the tableau whose columns  are obtained by concatenating the columns  of
$\overline{T}$ over $S$ below and by  $S\backslash T$ the tableau
whose  rows are obtained by concatenating the rows of $\overline{T}$
over $S$ from the right. Then by \cite[Thm. 4.2]{Taskin}
$$ S \ast T = \sum_{\substack{R \in SYT_n:\\
                 S\backslash T \leq_{weak} R \leq_{weak} S/T}} R$$
Namely the product structure can be read on the weak order poset of standard Young tableaux.

\begin{example}
Let $ S= \scriptstyle{\begin{array}{cc} 1&2 \\ 3 & \end{array}}$ and
  $T =\scriptstyle{\begin{array}{c} 1 \\ 2 \end{array}}$.
Then  $S\backslash T=\scriptstyle{\begin{array}{ccc} 1&2&4 \\ 3 & 5
\end{array}}$, ~ $S/ T=\scriptstyle{\begin{array}{cc} 1&2 \\ 3& \\ 4&
\\ 5
\end{array}}$.  Then
%\begin{equation}\label{P-R-example}
$$\begin{aligned}
PC_{\footnotesize{\begin{array}{ll} 1&2 \\ 3&
\end{array}}}~\ast~
    PC_{\footnotesize{\begin{array}{l} 1\\ 2  \end{array}}}
&=\shf(312,54) + \shf(132,54)\\
&= PC_{\footnotesize{\begin{array}{lll} 1&2&4\\3&5
\end{array}}}+
 PC_{\footnotesize{\begin{array}{lll} 1&2&4\\3\\ 5 \end{array}}}+
 PC_{\footnotesize{\begin{array}{ll} 1&2\\3&4\\ 5 \end{array}}}+
 PC_{\footnotesize{\begin{array}{ll} 1&2\\3\\4\\5 \end{array}}}.
\end{aligned}$$
%\end{equation}
On the other hand one can check from
 Figure~\ref{figure1}  that the
product $S \ast T$ is equal to the sum of all tableaux in the interval
$[S\backslash T, S/T]$.
\end{example}

By using the facts  that $(S\backslash T)^{evac}=T^{evac}\backslash S^{evac}$  and  $(S/T)^{evac}=T^{evac}/S^{evac}$ and  Proposition~\ref{order.preserving.maps},  one can easily deduce the following corallary to  Conjecture~\ref{main.conjecture}.

\begin{corollary} Let $S,S',T,T'$ be standard Young tableaux satisfying
$$\sh(S)=\sh(S') ~\text{  and }~  \sh(T)=\sh(T').$$  Then the intervals of the weak order $[S\backslash T, S/T]$ and  $[S'\backslash T', S'/T']$ are isomorhic. Equivalently,  the shuffle product  $S \ast T$ is determined by the shapes of the tableaux  rather than the tableaux itself.
\end{corollary}

\section{The cases where the conjecture holds}

\begin{lemma}
\label{conjecture.lemma1} Suppose that  $S\leq_{weak}T$  is a
covering relation in $SYT_n^{R}$ and $R,R'\in SYT_k$ has the same
shape. If $\st(S_{[k+1,n]})\lneq_{weak} \st(T_{[k+1,n]})$ then
$\mathcal{V}_{[R,R']}(S)\lessdot_{weak}\mathcal{V}_{[R,R']}(T)$.
\end{lemma}

\begin{proof}It is enough to consider the case when  $R$ and $R'$ differ by only one  dual Knuth relation determined by the triples $\{i,i+1,i+2\}$ for some
 $i\leq k-2$. Since $ S<_{weak}T$ is a covering relation, there exist $\sigma \in
\kc_S$ and $\tau \in \kc_T$ such that for  some $i<n$,
 $$ \begin{aligned}
 \sigma=a_1\ldots a_j a_{j+1}\ldots a_n  &\leq  a_1\ldots a_{j+1} a_{j}\ldots a_n=\tau, \text{ where } a_j< a_{j+1}
\end{aligned}
 $$
 i.e., $\sigma<\tau$ is also
covering relation the right weak order on $S_n$.  On the other hand by Lemma~\ref{inner.restriction.lemma} we have
 $$
 I(\sigma_{[k+1,n]}) =S_{[k+1,n]} \text{ and } T_{[k+1,n]}=I(\tau_{[k+1,n]})
 $$
Furthermore  the assumption  $S_{[k+1,n]} \lneq_{weak} T_{[k+1,n]}$  yields that  $\sigma_{[k+1,n]}\lneq\tau_{[k+1,n]}$ and $\sigma_{[1,k]}=\tau_{[1,k]}$.

 Now applying  the dual Knuth relation determined by the triple $\{i,i+1,i+2\}$  on $\sigma$
and $\tau$ gives two new permutations say $\sigma'$ and $\tau'$
such that $\sigma'\lessdot \tau'$ in the right weak Bruhat order and
therefore
 $$\mathcal{V}_{[R,R']}(S)= I(\sigma') ~\leq_{weak}~I(\tau') =\mathcal{V}_{[R,R']}(T).$$
Now if there exist a tableau $Q \in SYT_n^{R'}$ satisfying
$\mathcal{V}_{[R,R']}(S) \lneq_{weak}  Q \lneq_{weak}
\mathcal{V}_{[R,R']}(T)$ then  we have $$S \lneq_{weak}
\mathcal{V}_{[R',R]}(Q) \lneq_{weak} T$$ which is clearly a
contradiction.  Hence $\mathcal{V}_{[R,R']}(S) \lessdot_{weak}
\mathcal{V}_{[R,R']}(T)$.
\end{proof}

\begin{lemma}
\label{conjecture.lemma2} Suppose $$R=\tableau{{1}&{3}\\{2}}
~~\text{and}~~R'=\tableau{{1}&{2}\\{3}}.$$ Then $S\lessdot_{weak}T$
in $SYT_n^{R}$  if and only if
$\mathcal{V}_{[R,R']}(S)\lessdot_{weak}\mathcal{V}_{[R,R']}(T)$  in
$SYT_n^{R'}$.
\end{lemma}
\begin{proof} Since  $ S<_{weak}T$  there exist $\sigma
\in \kc_S$ and $\tau \in \kc_T$ such that     for some
$1\leq j<n$ we have
 $$
 \sigma=a_1\ldots a_j a_{j+1}\ldots a_n ~\text{ and }~\tau =a_1\ldots a_{j+1} a_{j}\ldots a_n, \text{ where } a_j< a_{j+1}.
 $$
 Observe that since $S$ and $T$ have the same  inner tableau $R$, we have $\{\sigma_{[1,3]},\tau_{[1,3]}\}\subset
 \kc_R=\{213,231\}$.

 If $\{a_j, a_{j+1}\}\not=\{1,3\}$ then applying
 the dual Knuth relation determined by $\{1,2,3\}$ on $\sigma$ and
 $\tau$ yields two permutations $\sigma' \in \mathcal{V}_{[R,R']}(S)$ and
 $\tau'\in \mathcal{V}_{[R,R']}(T)$ which still have
 $\sigma'\lessdot \tau'$ in the right weak order. Therefore
 $\mathcal{V}_{[R,R']}(S)\leq_{weak}\mathcal{V}_{[R,R']}(T)$ and it
 must be a covering relation.

Suppose  $\{a_j, a_{j+1}\}=\{1,3\}$.  Since $\kc_R=\{213,231\}$,  the insertion taleau $I(a_1\ldots a_{j-1})$ must  have the number $2$ its first row
left most position. Let for some $i\leq j-1$
$$b_1\ldots b_{i-1}~2~b_{i+1} \ldots
b_{j-1}$$ be the row word of $I(a_1\ldots a_{j-1})$ obtained by
reading numbers in each row of  $I(a_1\ldots a_{j-1})$ from left to right,  starting from
last row. Therefore the sequence $2~b_{i+1} \ldots b_{j-1}$ labels
the first row and moreover
$$\begin{aligned}
&b_1\ldots b_{i-1}~2~b_{i+1}b_{i+2} \ldots b_{j-1}~1~3~a_{j+2}\ldots a_n \in \kc_S \\
 &b_1\ldots b_{i-1}~2~b_{i+1}b_{i+2} \ldots b_{j-1}~3~1~a_{j+2}\ldots a_n\in \kc_T.
\end{aligned}$$
On the other hand  since $2<b_{i+1}< \ldots <b_{j-1}$ we have
$$\begin{aligned}
&2~b_{i+1}b_{i+2} \ldots b_{j-1}1~3~\Knuth ~2~b_{i+1}1~3~b_{i+2} \ldots b_{j-1}\\
&2~b_{i+1} b_{i+2}\ldots b_{j-1}1~3~\Knuth  ~b_{i-1}b_{i+1}b_{i+2}3~1~2 \ldots b_{j-1}
\end{aligned}$$
and moreover
$$\begin{aligned}
&b_1\ldots b_{i-1}2~b_{i+1}1~3~b_{i+2} \ldots b_{j-1}a_{j+2}\ldots a_n \in \kc_S \\
&b_1\ldots b_{i-1}b_{i+1}2~1~3~b_{i+2} \ldots b_{j-1}a_{j+2}\ldots a_n
\in \kc_T.\end{aligned}
$$
On the other hand applying dual Knuth relation determined by
$\{1,2,3\}$ we get
$$b_1\ldots b_{i-1}~3~b_{i+1}1~2~b_{i+2}  \ldots b_{j-1}a_{j+2}\ldots a_n \in \kc_{\mathcal{V}_{[R,R']}(S)} \text{
and } b_1\ldots b_{i-1}b_{i+1}~3~1~2~b_{i+2}  \ldots b_{j-1}a_{j+2}\ldots a_n
\in \kc_{\mathcal{V}_{[R,R']}(T)}
$$
which are clearly the generator of
$\mathcal{V}_{[R,R']}(S)\lessdot_{weak} \mathcal{V}_{[R,R']}(T)$.

Suppose now $ S<_{weak}T$ is a covering relation in $SYT_n^{R'}$.
Since  $R^t=R'$, by Proposition \ref{order.preserving.maps} we have
 $$ T^t\lessdot_{weak}S^t \text{ in }
SYT_n^{R}.by Proposition \ref{order.preserving.maps}$$ Now by the
previous result we have $\mathcal{V}_{[R,R']}(T^t)\lessdot_{weak}
\mathcal{V}_{[R,R']}(S^t)$ and therefore
$$\mathcal{V}_{[R',R]}(S)=(\mathcal{V}_{[R,R']}(S^t))^t\lessdot_{weak}
(\mathcal{V}_{[R,R']}(T^t))^t=\mathcal{V}_{[R,R']}(T).$$
\end{proof}

\begin{proposition}\label{two.rows.proposition} Suppose that  $S\lessdot_{weak}T$   in $SYT_n^{R}$ where
 $R \in SYT_k$ has exactly two rows.  If $R'$ is another tableau in $SYT_k$ having the same shape  with $R$ then
 $\mathcal{V}_{[R,R']}(S)\lessdot_{weak}\mathcal{V}_{[R,R']}(T)$ in $SYT_n^{R'}$.
\end{proposition}

\begin{proof} Suppose that $S\leq_{weak}T$
 is a covering relation in $SYT_n^{R}$ and $R\in SYT_k$ has two rows.
 When $k<3$ there is nothing to prove. For $k=3$ the only case that needs to be explored is when
 $R$ has non vertical or non horizontal shape, hence Lemma \ref{conjecture.lemma2} gives the desired result.

So we suppose the statement is true for $k-1$ and let $R\in SYT_k$.  It is enough to consider the case when
 $R$ and $R'$ differ by only
one dual Knuth relation determined by the triple $\{i,i+1,i+2\}$
where $i+2= k$. If $i+2<k$ then $S_{[1,i+2]}=T_{[1,i+2]}$ of $R$ has
still two rows and  induction  gives the desired result.  If $i+2=k$
then    we have the following classes of possibilities for the
tableau $R$:
\begin{equation}\label{row.lemma}
\begin{aligned}
&\tableau{{*}&{*}&_{k-2}& _k \\{*}&{*}& _{k-1}}~&~
&\tableau{{*}&{*}&_{k-2}& _{k-1} \\{*}&{*}& _{k}}~&~
&\tableau{{*}&{*}&{*}& _{k-2 }\\{*}&_{k-1}& _{k}}~&~
&\tableau{{*}&{*}&{*}& _{k-1 }\\{*}&_{k-2}& _{k}}&\\
&~~~~(a)&&~~~~(b)&&~~~~(c)&&~~~~(d)&
\end{aligned}
\end{equation}
Observe that in the last  two classes the dual Knuth relation
determined by $\{k-2,k-1,k\}$ interchanges the places of  $k-1$ and
$k-2$ and so  they refer to the cases with smaller inner tableau
$S_{[1,k-1]}=T_{[1,k-1]}$ and  the induction argument gives the
required  result.

For the first two classes  we have the following analysis:  Since $
S\lessdot_{weak}T$  there exist $\sigma \in \kc_S$ and $\tau \in
\kc_T$ such that  $\sigma<\tau$ is also a covering relation the
right weak order on $S_n$ i.e.,   for some $1\leq j<n$ we have
 $$
 \sigma=a_1\ldots a_j a_{j+1}\ldots a_n ~\text{ and }~\tau =a_1\ldots a_{j+1} a_{j}\ldots a_n, \text{ where } a_j< a_{j+1}.
 $$

 If $\{a_j, a_{j+1}\}\not=\{k,k-2\}$ then applying
 the dual Knuth relation determined by $\{k,k-1,k-2\}$ on $\sigma$ and
 $\tau$ yields two permutations $\sigma' \in \mathcal{V}_{[R,R']}(S)$ and
 $\tau'\in \mathcal{V}_{[R,R']}(T)$ which still have
 $\sigma'\lessdot \tau'$ in the right weak order. Therefore
 $\mathcal{V}_{[R,R']}(S)\leq_{weak}\mathcal{V}_{[R,R']}(T)$ and it
 must be a covering relation.

Now let  $\{a_j, a_{j+1}\}=\{k,k-2\}$, i.e.,
\begin{equation} \label{row.lemma2}
 \sigma=a_1\ldots~ a_{j-1} (k-2)~k ~a_{j+2}\ldots a_n ~\text{ and }~\tau =a_1\ldots a_{j-1} ~k~(k-2)~a_{j+2}\ldots a_n
 \end{equation}

\noindent {\it Case 1.}  We first consider the case illustrated in
\eqref{row.lemma}-$(a)$, where   $k-1$ comes before $k$ in every
permutations in the Knuth classes of $S$ and $T$.  Therefore the
tableau $I(a_1\ldots a_{j-1})$ must have the number $k-1$ and it
must be located  in the first row, since otherwise  the number $k$
drops to the second row of $T$ at the end of the insertion  of
$\tau$ and that is clearly a contradiction.   Let for some $r\leq
j-1$
$$b_1\ldots b_{r-1}~(k-1)~b_{r+1}b_{r+2} \ldots
b_{j-1}$$ be the row word of $I(a_1\ldots a_{j-1})$ obtained by
reading numbers in each row of  $I(a_1\ldots a_{j-1})$ from left to
right,  starting from last row. Therefore $(k-1)~b_{r+1}b_{r+2}
\ldots b_{j-1}$ lies on the first row and so $k-1<b_{r+1}<b_{r+2}<
\ldots <b_{j-1}$.

Now it is easy to see that
$$
b_1\dots b_{r-1}(k-1)b_{r+1}b_{r+2} \ldots
b_{j-1}(k-2)k~a_{j+2}\ldots a_n ~\Knuth~ b_1\ldots
b_{r-1}(k-1)b_{r+1}(k-2)k~ \ldots b_{j-1}a_{j+2}\ldots a_n
$$
lies in the Knuth class $S$ where as
$$b_1\ldots
b_{r-1}(k-1)b_{r+1}b_{r+2} \ldots b_{j-1}~k(k-2)a_{j+2}\ldots
 a_n ~\Knuth~ b_1\ldots b_{r-1}b_{r+1}(k-1)(k-2)k \ldots b_{j-1}a_{j+2}\ldots a_n
$$
lies in the Knuth class of $T$. Moreover  applying dual Knuth
relation determined by $\{k,k-1,k-2\}$ on the latter  permutations
we get
$$\begin{aligned}
b_1\ldots b_{r-1}~k~b_{r+1}(k-2)~(k-1)~ \ldots b_{j-1}a_{j+2}\ldots a_n \in \kc_{\mathcal{V}_{[R,R']}(S)} \\
 b_1\ldots b_{r-1}~b_{r+1}~k~(k-2)~(k-1) \ldots b_{j-1}a_{j+2}\ldots a_n
\in \kc_{\mathcal{V}_{[R,R']}(T)}
\end{aligned}
$$
and therefore $\mathcal{V}_{[R,R']}(S)\leq_{weak}
\mathcal{V}_{[R,R']}(T)$.

\noindent {\it Case 2.}  For the case illustrated in
\eqref{row.lemma}-(b),   We have
$$ T^t\leq_{weak} S^t \in SYT_N^{R^t} $$
where $R^t$ is a tableau of at most two columns having $k-1$ in its first column and $k$ in its second column.  Therefore we obtained
 $\sigma^t \in \kc_{S^t}$ and $\tau^t \in \kc_{T^t}$  by reversing $\sigma$ and $\tau$ of \eqref{row.lemma2} .i.e.,
$$
 \sigma^t=a_n\ldots (k-1)\ldots a_{j+2}~k~(k-2)~a_{j-1}\ldots a_1 ~\text{ and }~\tau^t =a_n\ldots (k-1)\ldots  a_{j+1} ~(k-2)~k~a_{j-1}\ldots a_1
 $$
Now consider the tableau  $I(a_n\ldots (k-1)\ldots a_{j+2})$.
Suppose first that  the  left most cell in  its  first row is
labeled by a number,  say $x$, which is smaller then $k-1$.  This
implies that  insertion of  the sequence  $(k-2)k$ in to
$I(a_n\ldots (k-1)\ldots a_{j+2})$ places the sequence $(k-2)k$ to
the right of $x$ but this contradicts to the fact that the inner
tableau $R^t$ has at most two columns. Therefore we have  $x=k-1$.
Now let for some $r\leq n- j-1$
$$b_1\ldots b_{r-1}~(k-1)~b_{r+1} \ldots
b_{n-j-1}$$ be the row word of $I(a_n\ldots (k-1)\ldots a_{j+2})$.
 Therefore $(k-1)~b_{r+1}b_{r+2}
\ldots b_{n-j-1}$ lies on the first row and so $k-1<b_{r+1}<b_{r+2}<
\ldots <b_{n-j-1}$ which yields
$$
b_1\ldots b_{r-1}~(k-1)~b_{r+1} \ldots
b_{n-j-1}~k~(k-2)~a_{j-1}\ldots a_1~\Knuth~ b_1\ldots
b_{r-1}~b_{r+1}~(k-1)~(k-2)~k~ \ldots b_{n-j-1}a_{j-1}\ldots a_1
$$
lies in $\kc_{S^t}$ whereas
$$ b_1\ldots b_{r-1}~(k-1)~b_{r+1} \ldots b_{n-j-1}~(k-2)~k~a_{j-1}\ldots
a_1~\Knuth~ b_1\ldots b_{r-1}~(k-1)~b_{r+1}~(k-2)~k \ldots
b_{n-j-1}a_{j-1}\ldots a_1
$$
lies in $\kc_{T^t}$. Now  reversing and then  applying dual Knuth
relation determined by $\{k,k-1,k-2\}$ on the latter  permutations
we get

$$\begin{aligned}
&a_1\ldots a_{j-1}b_{n-j-1}\ldots (k-1)(k-2)k~b_{r+1}b_{r-1}\ldots
b_1
 \in \kc_{\mathcal{V}_{[R,R']}(S)} \\
&a_1\ldots a_{j-1}b_{n-j-1}\ldots (k-1)(k-2) b_{r+1}
~k~b_{r-1}\ldots b_1 \in \kc_{\mathcal{V}_{[R,R']}(T)}
\end{aligned}
$$
and therefore $\mathcal{V}_{[R,R']}(S)\leq_{weak}
\mathcal{V}_{[R,R']}(T)$.

Lastly the fact  that resulting relations are in fact covering
relations   follows directly.

\end{proof}

\begin{corollary} Suppose that  $S\leq_{weak}T$  is a covering relation in $SYT_n^{R}$ where $R \in SYT_k$ has exactly two colums.
 If $R'$ is another tableau in $SYT_k$ having the same shape  with $R$ then    $\mathcal{V}_{[R,R']}(S)\leq_{weak}\mathcal{V}_{[R,R']}(T)$ is also a coving relation in $SYT_n^{R'}$.
\end{corollary}

\begin{proof} By Proposition \ref{order.preserving.maps} we have $T^t\leq_{weak}S^t$ in   $SYT_n^{R^t}$ where $R^t$ has exactly two rows.
Now $(R')^t$ has the same shape with $R^t$ and by previous theorem
 $\mathcal{V}_{[R^t,(R')^t]}(T^t)\leq_{weak}\mathcal{V}_{[R^t,(R')^t]}(S^t)$. Therefore
  $$\mathcal{V}_{[R,R']}(S)=(\mathcal{V}_{[R^t,R'^t]}(S^t))^t\leq_{weak}(\mathcal{V}_{[R^t,R'^t]}(T^t))^t=\mathcal{V}_{[R,R']}(T).$$
  \end{proof}

\begin{definition} For $T\in SYT_n$  and $A$ is  a corner cell of
$T$, denote by
$$T^{\uparrow A}~~ \text{ and }~~ \eta(T^{\uparrow A}) $$
the tableau obtained by applying reverse insertion algorithm to $T$
through the corner cell $A$ and  respectively the number which
leaves the tableau at the end.
\end{definition}

The following result on the hook shape tableaux is easy to deduce by
using reverse RSK algorithm.

\begin{lemma}\label{hook.lemma2} Let $R \in SYT_k$ be a tableau of hook shape with more then two rows and two columns and suppose that  the only two corner cells of $R$, say $A$ and $B$   are labeled by
$k$ or $k-1$.   Then $$\eta(R^{\uparrow A})\not = \eta(R^{\uparrow
B})$$ and if  $a_1\ldots a_k$ and $b_1\ldots b_k$ be two
permutations in the Knuth class of $R$ with $a_k=b_k$ then
$$I(a_1\ldots a_{k-1})=I(b_1\ldots b_{k-1}).$$
\end{lemma}

\begin{proof} Since the tableaux required to have more then two rows and two
columns it is enough the consider  the following tableaux  together
with their transposes, where $k$ labels the cell $A$ and $k-1$
labels the cell $B$.
\begin{equation}\label{hook.lemma}
\tableau{{*}&_{k-2}& _k \\{*}\\ _{k-1}}~ \mbox{  }~~ \tableau{{*}&
{*}& _k
\\_{k-2}\\ _{k-1}}
\end{equation}

Clearly $\eta(R^{\uparrow A})\not = \eta(R^{\uparrow B})$ and this
shows that if   $a_k=b_k$ then they must leave the tableau at the
end of a reverse insertion applied on the same corner cell, say $A$.
Therefore $ I(a_1\ldots a_{k-1})=R^{\uparrow A}=I(b_1\ldots
b_{k-1})$.

\end{proof}

\begin{proposition} Suppose that  $S\leq_{weak}T$  is a covering relation in $SYT_n^{R}$ where
$R \in SYT_k$ has hook shape. If $Q$ is another tableau in $SYT_k$
having the same shape  with $R$ then
  $\mathcal{V}_{[R,Q]}(S)\leq_{weak}\mathcal{V}_{[R,Q]}(T)$ is also a coving relation in $SYT_n^{Q}$.
\end{proposition}

\begin{proof}
Here we just need to deal with the case when
 $R$ is not a  horizontal or  a vertical tableau. On the other hand  for $k\leq 4$ the only non horizontal
 or vertical hook shape tableaux have either two rows or two columns and  Proposition \ref{two.rows.proposition}
 gives the required result.
 Therefore  in the rest  we assume that $n>k>4$ and $R$ has more then two rows and two columns.

We may assume that $R$ and $Q$ differ by only one dual Knuth
relation determined by the triple  $\{i,i+1,i+2\}$. If $i+2<k$ then
the subtableau $S_{[1,i+2]}=T_{[1,i+2]}$ of $R$  has still hook
shape and induction gives the desired result. So let   $R$ and $R'$
differ by a single dual Knuth relation determined by
$\{k-2,k-1,k\}$. Since $R$ has a hook shape this implies the dual
Knuth relation interchanges the places of $k$ and $k-1$ i.e., the
only two corner cells of $R$ are occupied by $k$ and $k-1$.

Now since $ S<_{weak}T$ is a covering relation, there exist $\sigma
\in \kc_S$ and $\tau \in \kc_T$ such that  $\sigma<\tau$ is also a
covering relation the right weak order on $S_n$ i.e.,   for some
$1\leq j<n$ we have
 $$
 \sigma=a_1\ldots a_j a_{j+1}\ldots a_n ~\text{ and }~\tau =a_1\ldots a_{j+1} a_{j}\ldots a_n, \text{ where } a_j< a_{j+1}.
 $$

If $\{a_j, a_{j+1}\}\not=\{k,k-2\}$ the result follows as discussed
in the proof of Proposition~\ref{two.rows.proposition}. So in the
rest we assume that $\{a_j, a_{j+1}\}=\{k,k-2\}$.

Observe that we have either  $\sigma_1=a_1=\tau_1$ or
$\sigma_n=a_n=\tau_n$. WL.O.G  assume $\sigma_n=a_n=\tau_n$ (the
first one can be deal with the same method on the transposes of the
tableaux). Therefore there exist some corner cells say $A_S$ and
$A_T$ of $S$ and $T$ respectively such that
$$\begin{aligned}&\eta(S^{\uparrow A_S}) =a_n= \eta(T^{\uparrow
A_T})\\
&S'=S^{\uparrow A_S}=I(a_1\ldots a_j a_{j+1}\ldots a_{n-1}) \\
&T'=T^{\uparrow A_T}=I(a_1\ldots a_{j+1} a_{j}\ldots a_{n-1})
\end{aligned}$$

\noindent {\it Case 1.} If  $a_n>k$ then $S'$ and $T'$ have still
the same inner tableau $R$.  Moreover  since  $S'\leq_{weak}T'$, we
have by induction
$$\mathcal{V}_{[R,R']}(S')\leq_{weak}\mathcal{V}_{[R,R']}(T')
$$
and therefore
$$\mathcal{V}_{[R,R']}(S)=\mathcal{V}_{[R,R']}(S')^{\downarrow a_n}
\leq_{weak}\mathcal{V}_{[R,R']}(T')^{\downarrow a_n}=\mathcal{V}_{[R,R']}(T)$$

\noindent {\it Case 2.}  If  $a_n\leq k$ then   the number $a_n$
leaves the tableaux through the sub-tableau $R$ in both reverse
insertion $S'=S^{\uparrow A_S}$ and $T'=T^{\uparrow A_T}$. Recall
that the only two corners of $R$ is labeled by $k$ and $k-1$. This
result    by Lemma \ref{hook.lemma2} that $a_n$ leaves $R$ as a
result of the reverse insertion algorithm applied on the same corner
cell say $C$, i.e.,
$$a_n= \eta(R^{\uparrow C})$$
Recall that $R$ has more
than two rows and two columns which leaves us with  the following
possibilities:
\begin{equation}\label{hook.lemma}
\tableau{{*}&_{k-2}& _k \\{*}\\ _{k-1}}~ \mbox{  }~~ \tableau{{*}&
{*}& _k
\\_{k-2}
\\_{k-1} }~ \mbox{  }~~
\tableau{{*}&_{k-2}& _{k-1} \\{*}\\ _{k}}~ \mbox{  }~~ \tableau{{*}&
{*}& _{k-1}
\\_{k-2}
\\_{k} }
\end{equation}

In the first two cases of  \eqref{hook.lemma}, one can observe
easily that either $a_n=k$ or
 $a_n<k-2$, but $a_n=k$ contradicts to the assumption that $\{a_j, a_{j+1}\}=\{k,k-2\}$.
Therefore $a_n<k-2$ and the application of  dual Knuth relation
determined by $\{k-2,k-1,k\}$ to the inner tableau $R'=R^{\uparrow
C}$ gives $Q'=Q^{\uparrow C}$.  Now by induction we have
$\mathcal{V}_{[R',Q']}(S')^{\downarrow a_n}
\leq_{weak}\mathcal{V}_{[R',Q']}(T'] $ and therefore
$$\mathcal{V}_{[R,Q]}(S)=\mathcal{V}_{[R',Q']}(S')^{\downarrow a_n}
\leq_{weak}\mathcal{V}_{[R',Q']}(T')^{\downarrow
a_n}=\mathcal{V}_{[R,Q]}(T).$$

In the last  two cases  \eqref{hook.lemma}, we have either either
$a_n=k-1$ or
 $a_n<k-2$. If  $a_n<k-2$ then the required result follows as
 above.  On the other hand if  $a_n=k-1$ then for  the tableau $R'=R^{\uparrow
C}$ we have the following possibilities
$$
\tableau{{*}&_{k-2} \\{*}\\ _{k}}~ \mbox{  }~~ \tableau{{*}& {*}
\\_{k-2}
\\_{k} }
$$
where in both cases every permutation in the Knuth class  have the
subsequence $k(k-2)$. This shows that any two permutations in the
Knuth class of $R$ that ends with $k-1$ must have the subsequence
$k(k-2)$ and this again contradicts to the fact that $\{a_j,
a_{j+1}\}=\{k,k-2\}$.

\end{proof}


\begin{thebibliography}{99}

\bibitem{Barbash-Vogan}
D. Barbash, D. Vogan, Primitive ideals and orbital integrals in
complex exceptional groups, {\it J. Algebra}, {\bf 80}, (1983),
350--382.

\bibitem{Bjorner} A. Bj\"orner,
 Topological Methods, {\it Handbook of Combinatorics,
 (R. Graham, M. Gr\"oschel and L. Lov\'asz, eds.)}, Elsevier, Amsterdam,
 (1995), 1819--1872.

\bibitem{Bjorner2}
A. Bj\"orner, M. L. Wachs, Permutation Statistics and Linear
Extensions of Posets, {\it J. Combin. Theory Ser. A}, {\bf 58},
no.1, (1991), 85--114.


\bibitem{Geck}
M. Geck,  On the induction of Kazhdan--Lusztig cells, {\it Bull.
London Math. Soc.}, {\bf 35} (2003), 608--614.

\bibitem{Greene}
C. Greene, An extension of Schensted's theorem, {\it Adv. in Math.},
{\bf 14} (1974), 254--265.



\bibitem{Kazhdan-Lusztig}
D. Kazhdan, G. Lusztig, Representations of Coxeter groups and Hecke
algebras, {\it Invent. Math.}, {\bf 53} (1979), 165--184.

\bibitem{Knuth} D. E. Knuth, {\it The art of computer programming}
Vol.3, Addison-Wesley (1969), 49--72.



\bibitem{Knuth1}
D. E. Knuth, Permutations, matrices and generalized Young tableaux,
{\it Pacific J. Math.}, {\bf 34} (1970), 709--727.



\bibitem{Loday-Ronco} J. Loday, M. Ronco, Order structure on the
algebra of permutations and of planar binary trees, {\it J.
Algebraic Combin.}, {\bf 15} (2002), 253--270.


\bibitem{Lusztig2}
G. Lusztig, {\it Characters of reductive groups over a finite
filed}, Ann. Math. Stud.  107 (Princeton  University Press, 1984).



\bibitem{Malvenuto-Reutenauer}
C. Malvenuto, C.Reutenauer, Duality between Quasi-Symmetric
Functions and the Solomon Descent algebra, {\it J. Algebra}, {\bf
177} (1995), 967--982.


\bibitem{Melnikov1}
A. Melnikov, On orbital variety closures of $\mathfrak{sl}_{n}$, I.
Induced Duflo order,
 math.RT/0311472. {\it J. Algebra}, to appear.

\bibitem{Melnikov2}
A. Melnikov, On orbital variety closures of $\mathfrak{sl}_{n}$, II.
Descendants of a Richardson orbital variety, math.RT/0311474. {\it
J. Algebra,} to appear.

\bibitem{Melnikov3}
A. Melnikov, Irreducibility of the associated varieties of simple
highest modules in $\mathfrak{sl}_{n}$, {\it C.R.A.S.I}, {\bf 316},
(1993), 53--57.

\bibitem{Melnikov4}
A. Melnikov, On orbital variety closures of $\mathfrak{sl}_{n}$,
III. Geometric properties, math.RT/0507504. {\it J. Algebra,} to
appear.

\bibitem{Poirier-Reutenauer}
S. Poirier, C. Reutenuer, Alg\'ebres de Hopf de Tableaux, {\it Ann.
Sci. Math. Qu\'ebec}, {\bf 19}, (1995), no. 1, 79--90.



\bibitem{Sagan}
B.E. Sagan, {\it The Symmetric Group, Second edition}.
 Springer-Verlag, New York, Inc., (2001).

\bibitem{Schensted} C. Schensted, Longest increasing and decreasing
subsequences, {\it Canad. J. Math.}, {\bf 13}, (1961), 179--191.

\bibitem{Schutzenberger1}
M. P. Sch\"utzenberger, Quelques remarques sur une construction de
Schensted, {\it Math. Scand.}, {\bf 12}, (1963), 117--128.


\bibitem{Schutzenberger2}
M. P. Sch\"utzenberger, La correspondence de Robinson, {\it
Combinatoire et Repr\'esentation du Groupe Sym\'etrique}, Lecture
Notes in Math.,{\bf  579},  Springer, Berlin (1977), 59--135.

\bibitem{Taskin}
M. Taskin, Properties of four partial orders on standard Young
tableaux, {\it J. of Combin. Theory Ser. A.} (to appear).

\bibitem{Vogan}
D. Vogan, Ordering of the primitive spectrum of a semisimple Lie
algebra, {\it Math. Ann.} {\bf 248} (1980), 195--2003.



\end{thebibliography}
\end{document}